\documentclass[11pt]{article}
\usepackage{amsmath,amsthm,amsfonts,amssymb,amscd, amsxtra}
\usepackage{url}
\usepackage[margin=1in]{geometry}
\usepackage{cite}
\usepackage{color}
\usepackage{graphicx}
\usepackage{caption}
\usepackage{subcaption}
\newtheorem{theorem}{Theorem}[section]

\newtheorem{proposition}{Proposition}[section]
\newtheorem{remark}{Remark}
\newtheorem{definition}{Definition}[section]

\newtheorem{example}{Example}

\DeclareMathOperator{\diag}{diag}
\DeclareMathOperator{\sgn}{sgn}
\DeclareMathOperator{\Id}{Id}

\newcommand{\lng}{\langle}
\newcommand{\rng}{\rangle}

\newcommand{\R}{\mathbb R}

\newcommand{\NN}{\mathbb{N}}

\begin{document}

\title{On the convergence of iterative schemes for solving a piecewise linear system of equations}

\author{Nicolas F. Armijo \thanks{Department of Applied Mathematics, University of São Paulo, Brazil (e-mail: {\tt
nfarmijo@ime.usp.br}).  The author was supported by Fapesp grant 2019/13096-2.}  
\and
Yunier  Bello-Cruz\thanks{Northern Illinois University, USA (e-mail: {\tt
yunierbello@niu.edu}).  The author was supported in part by R\&A Grant from NIU.}
\and
Gabriel Haeser \thanks{Department of Applied Mathematics, University of São Paulo, Brazil (e-mail: {\tt
ghaeser@ime.usp.br}). The author was supported by CNPq and Fapesp grant 2018/24293-0.}
}
\date{\today}

\maketitle


\begin{abstract}
\noindent This paper is devoted to studying the global and finite convergence of the semi-smooth Newton method for solving a piecewise linear system that arises in cone-constrained quadratic programming problems and absolute value equations. We first provide a negative answer via a counterexample to a conjecture on the global and finite convergence of the Newton iteration for symmetric and positive definite matrices. Additionally, we discuss some surprising features of the semi-smooth Newton iteration in low dimensions and its behavior in higher dimensions. Secondly, we present two iterative schemes inspired by the classical Jacobi and Gauss-Seidel methods for linear systems of equations for finding a solution to the problem. We study sufficient conditions for the convergence of both proposed procedures, which are also sufficient for the existence and uniqueness of solutions to the problem. Lastly, we perform some computational experiments designed to illustrate the behavior (in terms of CPU time) of the proposed iterations versus the semi-smooth Newton method for dense and sparse large-scale problems. Moreover, we included the numerical solution of a discretization of the Boussinesq PDE modeling a two-dimensional flow in a homogeneous phreatic aquifer. 

\medskip



\end{abstract}

\section{Introduction}
We consider  the following  piecewise linear system:
\begin{equation}\label{eq:pwls}
		x^+ +Tx=b, 
\end{equation} 
where,  denoting by  $\R^{n \times n}$ the set of $n \times n$ matrices with real entries,  the data consists of  a   vector $b\in \R^n\equiv \R^{n \times 1}$ and  a  nonsingular  matrix $T\in\R^{n \times n}$. The  variable $x=(x_1, x_2 ,\ldots ,x_n)^T$ is  a vector in $\R^n$ and  $x^+$    is the projection of $x$ onto $\R^n_+$, which has the $i$-th component equal to $x_i^+=\max\{x_ i,0\}, i=1,\dots,n$.  
Some works dealing with problem \eqref{eq:pwls} and its generalizations include \cite{Chen:2010,Bello-Cruz:2017,Griewank:2015, Sun:2015,Mangasarian:2009a,Ferreira:2015,Barrios:2015}.  Solutions of equation \eqref{eq:pwls} are closely related to at least two important classes of well-known problems, such as the quadratic cone-constrained programming: 
\begin{equation}
\begin{matrix} \label{quad_problem}
\text{minimize} & \frac{1}{2}x^TQx+q^Tx, \\
\text{subject to} & x \in \R^n_+,
\end{matrix}
\end{equation} 
and the absolute value equation \cite{Mangasarian:2006,Mangasarian:2009a}:
\begin{equation}\label{ave}\hat Tx-|x|=\hat b.\end{equation}
Namely, the projection onto $\R^n_+$ of a solution of problem \eqref{eq:pwls} with $T = (Q-{\Id})^{-1}$ and $b =  Tq$ satisfies the linear complementarity problem given by the first order optimality conditions of \eqref{quad_problem} (see \cite{Barrios:2016} for details), while, on the other hand, for $\hat T= -2T-{\Id}$, and $\hat b=-2b$, noting that $x^+=\frac{x+|x|}{2}$ one can see that problems \eqref{eq:pwls} and \eqref{ave} are equivalent. Here $Q$ is a symmetric matrix and ${\Id}$ denotes the identity matrix. These relations attest to the importance of finding novel and efficient iterative procedures for solving equation \eqref{eq:pwls}.

In this paper, we first focus our attention on the {\it semi-smooth Newton method} for solving problem \eqref{eq:pwls}, which consists of specifying a particular generalized Jacobian of $F$ at $x$ in the problem of finding the zeroes of
\begin{equation} \label{eq:fucpw}
F(x):=x^+ + Tx- b,  \qquad \qquad  ~x \in \R^n.
\end{equation}
Namely, starting at the point  $x^{0}\in \R^n$,
the semi-smooth Newton iteration is defined by the following linear equation:
\begin{equation}\label{eq:newtonc2pw}
\left(P(x^{k}) +T\right)x^{k+1}=b,  \qquad k\in\NN,
\end{equation} where
\begin{equation}\label{P(x)}P(x) := {\rm diag}({\rm sgn}(x^+
)),\quad x \in \R^n. \end{equation} 
The above iteration was proposed in \cite{Brugnano:2009}, which was shown to be {\it globally} convergent to a solution of problem \eqref{eq:pwls} under suitable assumptions. It has been extensively studied in the literature for solving generalizations of equation \eqref{eq:pwls}; see, for instance, \cite{Bello-Cruz:2017,Mangasarian:2009a,Ferreira:2015,BelloCruz:2016b}. We emphasize that the global and linear convergence of iteration \eqref{eq:newtonc2pw} has been proved only under restricted assumptions related to the norm of the matrices  $(T+P(x))^{-1}$ for all $x\in \R^n$, which come from the study of \eqref{eq:newtonc2pw} as a contraction fixed point iteration. A promising and novel approach for establishing finite convergence of \eqref{eq:newtonc2pw} was proposed in Theorem $3$ of \cite{Barrios:2016} under the assumption that the rows of the matrices $(T+P(x))^{-1}$ for all $x\in \R^n$ have a definite sign, that is, in every row the entries have the same sign. It is worth noting that the number of matrices $P(x)$ with $x\in \R^n$ in \eqref{P(x)} is finite ($2^n$ to be precise). Hence, if the semi-smooth Newton method \eqref{eq:newtonc2pw} converges, this convergence will occur after finitely many steps. In the pursuit of weaker and verifiable sufficient conditions ensuring convergence of the sequence generated by \eqref{eq:newtonc2pw}, it was conjectured in \cite{Bello-Cruz:2017} that iteration \eqref{eq:newtonc2pw} converges after finitely many steps if $T$ is symmetric and positive definite. In this paper, we show that this conjecture is false with a counterexample. However, interestingly, we show that this assumption is enough to guarantee the existence and uniqueness of solutions to problem \eqref{eq:pwls}. Moreover, the inverse of $P(x)+T$ for all $x\in \R^n$ always exists, and thus the semi-smooth Newton method \eqref{eq:newtonc2pw} is well-defined. We will also show that although this method may cycle under this assumption, it can never cycle between only two points. In the second part of this paper, we propose two novel iterative processes inspired by the well-known Jacobi and Gauss-Seidel iterative methods for solving linear systems of equations. The main idea is to consider the Newtonian system \eqref{eq:newtonc2pw} and apply a Jacobi or a Gauss-Seidel step at each iteration $k$. The main advantage of doing so is that the iteration is computed by solving a diagonal or a triangular linear system of equations, which is considerably simpler than solving the linear system in \eqref{eq:newtonc2pw} to find $x^{k+1}$. Also, we are able to present new sufficient conditions for the convergence of these two proposed methods, which are related to the classical diagonal dominance and Sassenfeld's criterion. The existence and uniqueness of the solution for equation \eqref{eq:pwls} are also proved under both conditions. We then show with an example that the standard diagonal dominance is insufficient to ensure the existence of solutions to problem \eqref{eq:pwls}. Finally, numerical results show that the proposed methods are competitive in terms of CPU time if compared with the semi-smooth Newton method \eqref{eq:newtonc2pw}. The numerical illustration suggests that the proposed iterative methods become more efficient when the dimension is high, and the matrix is sparse. We finish the paper with an applied experiment with real data, where we solve and discuss the results of solving a piecewise linear equation raising from a discretization of the Boussinesq PDE \cite{Brugnano2008} modeling a two-dimensional flow in a homogeneous phreatic aquifer.


\subsection{Notations and preliminaries} \label{sec:int.1}
Next, we quickly present some notations and  facts used throughout the paper.
We write $\NN$ for the nonnegative integers $\{0,1,2,\ldots\}$.
The canonical inner  product in $\R^n$ will be denoted by
$\lng\cdot,\cdot\rng$  and the induced norm is $\|\cdot\|$.  
  For $x\in \R^n$, $\sgn(x)$ will denote a vector with components equal to $1$, $0$ or $-1$ depending on whether the corresponding component of the
vector $x$ is positive, zero or negative, respectively.    If $a\in\R$ and $x\in\R^n$, then denote $a^+:=\max\{a,0\}$, $a^-:=\max\{-a,0\}$ and $x^+$ and  $x^-$  the vectors with  $i$-th component equal to $x_i^+$ and  $x_ i^-$, respectively, $i=1,\dots,n$. Note that $x^+$ and $x^-$ are the projections of $x$ onto the cones $\R^n_+$ and $\R^n_-$, respectively. The matrix ${\Id} \in \R^{n \times n}$ denotes the  identity matrix.  If $x\in \R^n$ then $\diag (x)\in \R^{n \times n}$ will denote a  diagonal matrix with $(i,i)$-th entry equal to $x_i$, $i=1,\dots,n$. Denote $\|M\|:=\max \{\|Mx\|~:~  x\in \R^{n}, ~\|x\|=1\}$ for any  $M \in \R^{n\times n}$.
The following result is well-known and will be needed in the sequel:
\begin{theorem}[Contraction mapping principle \cite{Ortega:1987}, Thm. 8.2.2,  page 153] \label{fixedpoint}
 Let  $\Phi : \R^{n} \to   \R^{n}$.  Suppose that there exists $\lambda \in  [0,1)$ such that  $\|\Phi(y)-\Phi(x)\| \le \lambda \|y-x\|$, for all $ x, y \in  \R^{n}$. Then, there exists a unique $\bar x\in  \R^{n}$  such that $\Phi(\bar x) = \bar x$.
\end{theorem}

\section{The semi-smooth Newton method} \label{sec: defscls}

In this section, we present and analyze the convergence of the semi-smooth Newton method given by iteration \eqref{eq:newtonc2pw} for solving problem \eqref{eq:pwls}. In \cite{Barrios:2016,Brugnano:2009}, it was shown that the condition $\sgn((x^{k})^+) = \sgn((x^{k+1})^+)$ is sufficient to declare that $x^{k+1}$ is a solution of equation \eqref{eq:pwls}. We now show, in fact, that a component-wise version of this stopping criterion holds:


\begin{proposition}[Component-wise stopping criterion]\label{gen_pk=pk+1}
Assume that the sequence $(x^k)_{k\in\NN}$ generated by method \eqref{eq:newtonc2pw} is well defined. If $\sgn((x^{k+1}_i)^+) = \sgn((x^{k}_i)^+)$ for some $i \in \{1,\dots,n\}$ and some $k$, then $F_i(x^{k+1}) = 0$, where $F$ is defined in \eqref{eq:fucpw}.
\end{proposition}

\begin{proof}
By the definition of $F$ and $x^{k+1}$, we have
\begin{align*}
F(x^{k+1}) = & (x^{k+1})^+ + Tx^{k+1} - b \\
= & (P(x^{k+1})+T)x^{k+1} - b\\
= & (P(x^{k+1})+T)x^{k+1} - (P(x^{k})+T)x^{k+1}\\
= & (P(x^{k+1})-P(x^{k}))x^{k+1}.
\end{align*} Hence,
\begin{equation} \label{F_sgn_eq}
F_i(x^{k+1}) = \left(\sgn((x^{k+1}_i)^+)-\sgn((x^{k}_i)^+)\right)x^{k+1}_i,
\end{equation} for any $i \in \{1,\dots,n\}$, which implies the desired result.
\end{proof}

In our study, a crucial role will be played by diagonal dominance. Let us start by showing that in the most extreme case of diagonal dominance, namely when the matrix is diagonal, one can list all solutions of the equation, and iteration \eqref{eq:newtonc2pw} finds a solution in at most two steps.

\begin{proposition}[Finite convergence for the diagonal case]\label{diag_case}
Let $b \in \R^n$ and $T\in\R^{n\times n}$ be a diagonal matrix with entries $T = (t_{ii})$, $i=1,\dots,n$ such that $t_{ii} \notin \{0,-1\}$ for all $i$. Equation \eqref{eq:pwls} has no solutions if, and only if, $t_{ii}\in(-1,0)$ and $b_i<0$ for some $i$. If a solution of \eqref{eq:pwls} exists, then \eqref{eq:newtonc2pw} converges in at most two iterations to one of the solutions. In this case, the number of solutions of problem \eqref{eq:pwls} is given by $2^r$, where $r$ is the number of indexes $i$ such that $b_i>0$ and $t_{ii}\in(-1,0)$.
\end{proposition}

\begin{proof}
Let $x^0 \in \R^n$ be any starting point. By definition, we have
$$ x^1 = (P(x^0) + T)^{-1}b.$$ Since $T$ is a diagonal matrix, we have the following component-wise expression for $x^1$:
\begin{equation}\label{coord-i}
x^1_i = ((P(x^0) + T)^{-1}b)_i =
\begin{cases}
 \frac{b_i}{t_{ii}}, & x^0_i \leq 0, \\ 
 \frac{b_i}{1+t_{ii}}, & x^0_i > 0,    \end{cases}
\end{equation} for $i=1,\dots,n$. 
For a fixed $i$, if $t_{ii} > 0$, we get by \eqref{coord-i} that $\sgn((x^1_i)^+) = \sgn(b_i^+)$ and since $b$ is fixed we deduce that $\sgn((x^2_i)^+) = \sgn((x^1_i)^+)$. Similarly, if $t_{ii}<-1$, we conclude that $\sgn((x^1_i)^+)=1-\sgn(b_i^+)$. Hence, $\sgn((x^2_i)^+) = \sgn((x^1_i)^+)$. Thus, when there is no $i$ such that $t_{ii}\in(-1,0)$, by Proposition \ref{gen_pk=pk+1} we deduce that \eqref{eq:newtonc2pw} converges in two steps. In particular one can check that in this case the solution is unique with $i$-th component equals to 
\[ \left\{
\begin{matrix} 
      \frac{b_i}{t_{ii}}, & b_i \leq 0, \\
      \frac{b_i}{1+t_{ii}}, & b_i > 0,
   \end{matrix}\right.
\]
when $t_{ii}>0$, and
\[
\begin{cases}
\frac{b_i}{1+t_{ii}} ,& b_i \leq 0,\\
\frac{b_i}{t_{ii}} ,& b_i > 0,\\
\end{cases}
\]
when $t_{ii} < -1$.

Now, it directly follows from equation $\eqref{eq:pwls}$ that there is no solution if $b_i<0$ for some $i$ such that $t_{ii}\in(-1,0)$. If, however, $b_i\geq 0$ for such $i$, the solutions for each component-wise equation are given by $\frac{b_i}{t_{ii}}$ and $\frac{b_i}{1+t_{ii}}$, amounting to the desired formula for the number of solutions.
Therefore, assuming that the problem has a solution, we conclude that $b_i\geq0$ for all $i$ such that $t_{ii}\in(-1,0)$. By \eqref{coord-i}, it is easy to see that in this case, we have $\sgn((x^1_i)^+)=\sgn((x^0_i)^+)$. Using Proposition \ref{gen_pk=pk+1} and a similar computation done previously, we conclude that the method converges to a solution in at most two steps.
\end{proof}

An auxiliary result in our analysis follows next.

\begin{proposition} \label{prop_xpy-ypx}
Let $x,y \in \R^n$. Then
\begin{itemize}
\item[i)] $(y_i-x_i)(\sgn(y^+_i)x_i-\sgn(x^+_i)y_i) \leq 0, \quad \forall i = 1,\dots,n;$
\item[ii)] $(y_i-x_i)(\sgn(x^+_i)x_i-\sgn(y^+_i)y_i) \leq 0, \quad \forall i = 1,\dots,n.$
\end{itemize}
 In particular we have that $(y-x)^T(P(y)x-P(x)y) \leq 0$ and $(y-x)^T(P(x)x-P(y)y) \leq 0$.
\end{proposition}

\begin{proof} We only have three different cases to be analyzed:
\begin{itemize}
\item[1)] if $\sgn(y^+_i)=\sgn(x^+_i)=s,$ then
$ (y_i-x_i)(\sgn(y^+_i)x_i-\sgn(x^+_i)y_i) = -s(y_i-x_i)^2 \leq 0;$
\item[2)] if $\sgn(y^+_i)=1$ and $\sgn(x^+_i)=0,$ then
$ (y_i-x_i)(\sgn(y^+_i)x_i-\sgn(x^+_i)y_i) = x_i(y_i-x_i) \leq 0;$
\item[3)] if $\sgn(y^+_i)=0$ and $\sgn(x^+_i)=1,$ then
$ (y_i-x_i)(\sgn(y^+_i)x_i-\sgn(x^+_i)y_i) = -y_i(y_i-x_i) \leq 0.$
\end{itemize} Therefore, as a consequence we obtain that
$$(y-x)^T (P(x)y-P(y)x)= \sum_{i=1}^n (y_i-x_i)(\sgn(y^+_i)x_i-\sgn(x^+_i)y_i) \leq 0.$$ The second statement follows by a similar computation.
\end{proof}

Until the end of this section, we assume that the matrix $T$ is symmetric and positive definite. This assumption has been considered before in \cite{Bello-Cruz:2017} where the authors conjectured the global and finite convergence of the method under this assumption. To fully address this conjecture, we begin with a result about the existence and uniqueness of the solution of equation \eqref{eq:pwls} under this assumption. 
\begin{proposition}[Existence and uniqueness of solutions]\label{uniqueness-sol}
If $T$ is a symmetric and positive definite $n \times n$ matrix and $b \in \R^n$, then problem \eqref{eq:pwls} has one and only one solution.
\end{proposition}

\begin{proof}

Note first that the matrix $T + {\Id}$ is also symmetric and positive definite. Hence, $(T + {\Id})^{-1}$ exists. For any symmetric matrix $M$ we denote $\lambda_1(M) \leq \cdots \leq \lambda_n(M)$ the ordered eigenvalues of $M$. The matrix $T + {\Id}$ also satisfies $\lambda_i(T + {\Id}) = 1 + \lambda_i(T)$, $i=1,\dots,n$.
Since $\lambda_{1}(T)> 0$, we deduce that
\begin{align*}
\|(T + {\Id})^{-1}\|_2 &= \lambda_{n}((T + {\Id})^{-1}) \\
&= \frac{1}{\lambda_{1}(T + {\Id})} \\
&= \frac{1}{\lambda_{1}(T) + 1} \\
&< 1.
\end{align*} Defining the function $\Phi(x) := (T + {\Id})^{-1}(b-x^-)$ and using the decomposition $x = x^+ - x^-$ we can easily show that the solutions of problem \eqref{eq:pwls} coincide with the fixed points of $\Phi$. Moreover, $\Phi$ is a contraction and then has a unique fixed point. Indeed,
\begin{align*}
\| \Phi(x) - \Phi(y) \|_2 &= \|(T + {\Id})^{-1}(y^--x^-)\|_2 \\
&\leq \|(T + {\Id})^{-1}\|_2 \|(y^--x^-)\|_2 \\
&\leq \|(T +{\Id})^{-1}\|_2 \|x-y\|_2 \\
&=\frac{1}{\lambda_{1}(T) + 1} \|x-y\|_2,
\end{align*} 
using that $x^-$ and $y^-$, the projections of $x$ and $y$ onto $\R^n_{-}$ (a closed and convex set), are non-expansive in the last inequality. Thus, the result follows from Theorem \ref{fixedpoint}.\end{proof}
\begin{remark}[Another proof of the uniqueness of solutions]
The uniqueness of the solution of equation \eqref{eq:pwls} when $T$ is positive definite showed in Theorem \ref{uniqueness-sol} can also be deduced using the fact
\begin{equation} \label{eq:xpx-ypy}
(y-x)^T (P(x)x-P(y)y) \leq 0,
\end{equation}
as seen in Proposition \ref{prop_xpy-ypx}. Indeed, if there exist two solutions $x$ and $y$ of problem \eqref{eq:pwls}, we have by definition $x^++Tx = b$, 
$y^++Ty = b$. Subtracting those equations and multiplying the resulting equation by $(y-x)^T$, we get that
$$ (y-x)^T T(y-x) = (y-x)^T (x^+-y^+) \leq 0, $$ where the inequality follows from \eqref{eq:xpx-ypy} noting that $P(x)x=x^+$ and $P(y)y=y^+$. Thus, when $T$ is positive definite it must hold that $x=y$.  
\end{remark}

Due to the nature of iteration \eqref{eq:newtonc2pw}, we have that the sequence $(x^{k})_{k\in\NN}$ has only a finite number of different elements. This happens since the set $S := \{(P(x)+T)^{-1}b \text{ ; } x \in \R^n\}$ has at most $2^n$ different elements and $(x^{k})_{k\in\NN} \subseteq S$. The conclusion of this observation is that we have only two possible outcomes: the method converges in a finite number of steps, or it cycles. Note also that when $\bar{x}$ solves equation \eqref{eq:pwls}, we have $\bar{x}=(P(\bar{x})+T)^{-1}b$. Hence, by \eqref{eq:newtonc2pw}, $x^{k+1}=\bar{x}$ holds whenever $P(x^{k})=P(\bar{x})$. Thus, the problem amounts to finding a point in the same orthant as a solution. Let us first show that the method can only cycle among three or more points.

\begin{theorem} [Newton does not cycle between two points] \label{no-cycles-with-two-elements} Let $T\in\R^{n\times n}$ be a symmetric and positive definite matrix and $b\in\R^n$, then  
the sequence generated by the semi-smooth Newton method \eqref{eq:newtonc2pw} does not cycle between two points.
\end{theorem}
\begin{proof}
If the sequence generated by iteration \eqref{eq:newtonc2pw} cycles between two points $x$ and $y$, then 
$$ (P(x)+T)y = b,$$
$$ (P(y)+T)x = b.$$ So, we have that $T(y-x) = P(y)x-P(x)y$. Multiplying by $(y-x)^T$, we obtain from Proposition \ref{prop_xpy-ypx} that
$$ (y-x)^T T(y-x) = (y-x)^T(P(y)x-P(x)y) \leq 0,$$ which implies $x=y$.
\end{proof}

When $n\leq2$, let us show that the sequence generated by method \eqref{eq:newtonc2pw} in fact does not cycle.

\begin{theorem}[Finite convergence for low dimensions] \label{teo_finite_conv}
Let $b \in \R^n$ and $T\in\R^{n\times n}$ be a symmetric and positive definite matrix. Then, iteration \eqref{eq:newtonc2pw} has global and finite convergence for $n=1,2$.
\end{theorem}

\begin{proof}
The case $n=1$ is a direct consequence of Proposition \ref{diag_case}. When $n=2$, note that the sequence of Newton iterates has at most four different elements, however, no cycle of size four is possible, as these points are necessarily in four different quadrants of $\R^2$, coinciding in sign with the solution (which necessarily exists due to Proposition \ref{uniqueness-sol}), which implies convergence.

Now, suppose that there is a cycle among three different points $x,y,z$, different from $\bar{x}$, the unique solution. Clearly, these four points must lie in different quadrants of $\R^2$. Without loss of generality, let us assume that $x$ and $y$ lie in opposite quadrants, in the sense that $P(x)+P(y)=\Id$, and that the points satisfy the following equations: 
$$ (P(x)+T)y=b,$$
$$ (P(y)+T)z=b,$$
$$ (P(z)+T)x=b.$$ We have that $T(y-x) = P(z)x-P(x)y$, and therefore, multiplying by $(y-x)^T$, we obtain
\begin{align*}
(y-x)^T T(y-x) &= (y-x)^T(P(z)x-P(x)y)\\
&= (y_1-x_1)(\sgn(z_1^+)x_1-\sgn(x_1^+)y_1) + (y_2-x_2)(\sgn(z_2^+)x_2-\sgn(x_2^+)y_2).
\end{align*} Since $z$ is not in the opposite quadrant of $x$ we only have two options, $\sgn(z_1^+) = \sgn(x_1^+)$ and $\sgn(z_2^+) \neq \sgn(x_2^+)$, or $\sgn(z_1^+) \neq \sgn(x_1^+)$ and $\sgn(z_2^+) = \sgn(x_2^+)$. In both cases it can be checked that $(y-x)^T T(y-x) \leq 0$, which leads to a contradiction. The result now follows from Theorem \ref{no-cycles-with-two-elements}.
\end{proof}

We ran extensive numerical experiments in our pursuit of understanding the finite convergence of the semi-smooth Newton method \eqref{eq:newtonc2pw}. For symmetric and positive definite matrices $T\in\R^{n\times n}$ ($1000$ randomly generated problems for each dimension $n$), we recorded the number of iterations that \eqref{eq:newtonc2pw} needed to converge at each dimension. The results are shown below in Table \ref{tab:NumericalExperimentsGen}, where the percentage of problems solved for several different problem dimensions and the corresponding iterations are presented. 
\begin{table}[h!]
\centering
\begin{tabular}{|l|l|l|l|l|l|}
\hline
$n$/\%& 1 iter. & 2 iter. & 3 iter. & 4 iter. & 5 iter.\\ \hline
4 & 7.2 & 49.1 & 35.9 & 7 & 0.8\\ \hline
8 & 0.6 & 37.7 & 48.8 & 11.4 & 1.5\\ \hline
16 & 0 & 16.6 & 63.1 & 19.5 & 0.8\\ \hline
32 & 0 & 6 & 69 & 24.1 & 0.9\\ \hline
64 & 0 & 1.6 & 68.5 & 29.4 & 0.5\\ \hline
128 & 0 & 0.3 & 60.1 & 39.1 & 0.5\\ \hline
256 & 0 & 0 & 57.3 & 41.9 & 0.8\\ \hline
512 & 0 & 0 & 50.9 & 48.8 & 0.3\\ \hline
1024 & 0 & 0 & 43 & 56.9 & 0.1\\ \hline
2048 & 0 & 0 & 36.6 & 63.2 & 0.2\\ \hline
4096 & 0 & 0 & 30.6 & 69.2 & 0.2\\ \hline
\end{tabular}
\caption{Percentage of positive definite problems solved at each amount of iterations (iter.) for different dimensions ($n$).}
\label{tab:NumericalExperimentsGen}
\end{table}

We observe that the higher the dimension, the more iterations are needed. Surprisingly, the method performs exceptionally well as the required number of iterations grows very slowly (always capped by $5$) with respect to the growth of the dimension $n$. Indeed, for any $n\leq 4096$, no problem needed more than five iterations of \eqref{eq:newtonc2pw} to be solved.
\begin{table}[h!]
\centering
$$\begin{tabular}{|l|l|l|l|}
\hline
$n$/\% & 1 iter. & 2 iter. & 3 iter.\\ \hline
4 & 0 & 98.9 & 1.1\\ \hline
8 & 0 & 100 & 0\\ \hline
16 & 0 & 100 & 0\\ \hline
32 & 0 & 100 & 0\\ \hline
64 & 0 & 100 & 0\\ \hline
128 & 0 & 99.9 & 0.01\\ \hline
256 & 0 & 98.9 & 1.1\\ \hline
512 & 0 & 95 & 5\\ \hline
1024 & 0 & 86.6 & 13.4\\ \hline
2048 & 0 & 74.8 & 25.2\\ \hline
4096 & 0 & 54.9 & 45.1\\ \hline
\end{tabular}$$
\caption{Numerical experiments for ``almost'' diagonal positive definite matrices.}
\label{tab:NumericalExperimentsDiag}
\end{table}

In the tests illustrated in Table \ref{tab:NumericalExperimentsDiag}, the matrix $T$ is symmetric, positive definite, and ``almost'' diagonal, \emph{i.e.}, the elements of the diagonal are much greater compared with the off-diagonal entries. The diagonal entries are of the order of thousands, while the off-diagonal elements are smaller than one. Here, as suggested by Proposition \ref{diag_case}, no problem required more than three iterations to reach the solution.

Although the numerical tests based on random data suggest that the semi-smooth Newton method does not cycle, we were able to find a particular example of problem \eqref{eq:pwls} that shows that the Newton iteration \eqref{eq:newtonc2pw} may cycle among three points in $\R^3$ even when $T$ is symmetric and positive definite. This gives us a counterexample to the conjecture raised in \cite{Bello-Cruz:2017} on the global and finite convergence of the Newton iteration \eqref{eq:newtonc2pw} under this assumption.

\begin{example}[Counterexample on the finite convergence]
The semi-smooth Newton method \eqref{eq:newtonc2pw} fails to converge in the case $n=3$ with the following symmetric and positive definite matrix
\begin{equation}\label{data-conter}
T = \frac{1}{100}
\left(\begin{matrix}
32 & -26 & 21\\
-26 & 33 & -23\\
21 & -23 & 17
\end{matrix}\right)
\quad 
\mbox{and}\quad 
b = \frac{1}{100}
\left(\begin{matrix}
18\\
-48\\
30
\end{matrix}\right).
\end{equation}

It can be easily checked that the following points $x,y$ and $z$ conform a cycle of the method.
\[
x =
\left(\begin{matrix}
\frac{319}{1435}\\
\\
-\frac{1849}{6379}\\
\\
\frac{190}{1191}
\end{matrix}\right)
,\quad 
y =
\left(\begin{matrix}
-\frac{527}{2978}\\
\\
-\frac{1490}{923}\\
\\
-\frac{81}{2777}
\end{matrix}\right)
, \quad 
z =
\left(\begin{matrix}
-\frac{306}{95}\\
\\
\frac{18}{95}\\
\\
6
\end{matrix}\right).
\]
Indeed, those points satisfy the cycle equations
$$ (P(x)+T)y=b,$$
$$ (P(y)+T)z=b,$$
$$ (P(z)+T)x=b,$$ which proves the statement.\end{example}

\section{Jacobi-Newton and Gauss-Seidel-Newton methods}

Based on the well-known Jacobi and Gauss-Seidel methods for solving linear systems, we define and analyze two novel methods, which we call Jacobi-Newton and Gauss-Seidel-Newton methods, for solving problem \eqref{eq:pwls}. We start with two definitions related to classical diagonal dominance and Sassenfeld's criterion.

\begin{definition}[Strong diagonal dominance]\label{strong_diag_cond}
Let $T=(t_{ij})\in\R^{n\times n}$. We say that $T$ is strongly diagonal dominant if
$$ \frac{1}{| t_{ii} |}\left(1+\sum_{\substack{j=1 \\ j \neq i}}^n |t_{ij}| \right) <1,\quad\forall i=1,\dots,n.$$
\end{definition}

Note that if $T$ is strongly diagonal dominant, then $T$ is diagonal dominant. However, we need this stronger condition to ensure the global convergence of the Jacobi-Newton method, which will be presented later in \eqref{jaco_newton}. 

We now introduce a weaker condition for $T$, which is a variation of the classical Sassenfeld's condition. 

\begin{definition}[Strong Sassenfeld's condition]\label{strong_sassen_cond}
Let $T=(t_{ij})\in\R^{n\times n}$. Define $\beta_i$ and $\beta$ as follows
\begin{equation}\label{beta-1} \beta_1 := \frac{1}{| t_{11} |}\left(1+\sum_{j=2}^n |t_{1j}| \right),\end{equation}
\begin{equation}\label{beta-i} \beta_i := \frac{1}{| t_{ii} |}\left(\sum_{j=1}^{i-1}|t_{ij}|\beta_j+\sum_{j=i+1}^n|t_{ij}| + 1\right),\quad\forall i=2,\dots,n,\end{equation}
and
$$ \beta := \max_{i=1,\dots,n.}\beta_i. $$ We say that $T$ satisfies the strong Sassenfeld's condition if $\beta < 1$.
\end{definition}
It is easy to see that if $T$ is strongly diagonal dominant, then $T$ satisfies the strong Sassenfeld's condition. However, the converse implication is not true in general. Now we prove the existence and uniqueness of solutions for problem \eqref{eq:pwls} under the strong Sassenfeld's condition. This criterion will also be used later to prove the convergence of the Gauss-Seidel-Newton method, which we will introduce.

\begin{theorem}[Existence and uniqueness of solutions under strong Sassenfeld's condition]\label{exis_uniq_sassenf}
Let $b \in \R^n$ and $T\in\R^{n\times n}$. If $T$ satisfies the strong Sassenfeld's condition, then equation \eqref{eq:pwls} has one and only one solution.
\end{theorem}

\begin{proof}

To prove the existence and uniqueness of the solution, we will use the contraction mapping principle (Theorem \ref{fixedpoint}). Let us first decompose the matrix $T$ as a sum of $L+D+U$ where $D$ is the diagonal of $T$ and $L$ and $U$ are the strictly lower and upper parts of $T$, respectively.  We first define the mapping $\psi:\R^n\to \R^n$ such that 
$$ \psi(x) := (D+L)^{-1}(-Ux+b-x^+), \quad \forall x \in \R^n.$$
Then, we prove that the fixed points of $\psi$ are solutions of problem \eqref{eq:pwls} and also that it is a contraction. Indeed, note that $x=\psi(x)$ is equivalent to $(D+L)x=-Ux+b-x^+$, or $(D+L+U)x+x^{+}=b$. Hence, $x$ is solution of equation \eqref{eq:pwls} given that $T=D+L+U$.

On the other hand, to prove the contraction property of $\psi$ consider  any $x,y \in \R^n$ and define $u:=\psi(x)$ and $w:=\psi(y)$. Note that 
$$ u-w = \psi(x)-\psi(y) = (D+L)^{-1}(-U(x-y)+y^+-x^+),$$ which is equivalent to
$$ (D+L)(u-w) = -U(x-y)+y^+-x^+.$$
Hence,
\begin{equation*}
u_1-w_1 = \frac{1}{t_{11}}\left( -\sum_{j=2}^n t_{1j}(x_j-y_j) + y_1^+-x_1^+ \right),
\end{equation*} and after taking the absolute value in both sides and using its triangular inequality property, we have 
\begin{align*}
|u_1-w_1| &\leq \frac{1}{|t_{11}|}\left( \sum_{j=2}^n |t_{1j}||x_j-y_j| + |x_1^+-y_1^+| \right)\\
&\leq \frac{1}{|t_{11}|}\left( \sum_{j=2}^n |t_{1j}||x_j-y_j| + |x_1-y_1| \right)\\
&\leq \frac{1}{|t_{11}|}\left( \sum_{j=2}^n |t_{1j}| + 1\right) \|x-y\|_{\infty} \\
&= \beta_1 \|x-y\|_{\infty}.
\end{align*} The second inequality follows from the nonexpansiveness of the projections and the last equality by the definition of $\beta_1$ given in \eqref{beta-1}.
By induction let us assume that $|u_i-w_i| \leq \beta_i\|x-y\|_{\infty}$ for all $i=1,\dots,p-1,$ and we will show that $|u_p-w_p| \leq \beta_p\|x-y\|_{\infty}$. Analogously, 
\begin{equation*}
u_p-w_p = \frac{1}{t_{pp}}\left( -\sum_{j=1}^{p-1} t_{pj}(u_j-w_j) - \sum_{j=p+1}^n t_{pj}(x_j-y_j) + y_p^+-x_p^+ \right),
\end{equation*} which implies that
\begin{align*}
|u_p-w_p| &\leq \frac{1}{|t_{pp}|}\left(\sum_{j=1}^{p-1}|t_{pj}|\beta_j+1+\sum_{j=p+1}^n|t_{pj}| \right)\|x-y\|_{\infty}\\
&= \beta_p\|x-y\|_{\infty},
\end{align*} where we used the induction assumption and the nonexpansiveness of the projection, together with the definition of $\beta_p$ given in \eqref{beta-i} in the last equality.
Now, we are ready to prove the contraction property of $\psi$ by observing that for all $x,y \in \R^n$, we have
\begin{align*}
\| \psi(x)-\psi(y) \|_{\infty} &= \| u-w \|_{\infty} \\
&= \max_{i=1,\dots,n.}|u_i-w_i|\\
&\leq \beta\|x-y\|_{\infty}.
\end{align*} Thus, since $\beta < 1$, we have that $\psi$ is a contraction mapping and therefore has a unique fixed point.
\end{proof}

Clearly, if $T$ is strongly diagonal dominant, the result above is also true, \emph{i.e.}, problem \eqref{eq:pwls} has one and only one solution. A natural question that arises here is if this assumption can be relaxed to the classical diagonal dominance and if the existence and uniqueness of solutions would still hold for equation \eqref{eq:pwls}. Unfortunately, that is not true in general, as is shown in the following example.

\begin{example}[Diagonal dominance is not sufficient for existence of solutions] \label{teo_finite_conv_diagdom}
Let $b \in \R^n$ and $T$ a diagonal dominant matrix. Then, problem \eqref{eq:pwls} may fail to have solutions. Indeed, let us consider the following data 
\[
T = \frac{1}{100}
\left(\begin{matrix}
-26 & 16\\
23 & -33
\end{matrix}\right),
\quad
b = \frac{1}{100}
\left(\begin{matrix}
-12\\
12
\end{matrix}\right).
\]
Note that $T$ is diagonal dominant and the points $x,y$ and $z,w$ conform to two different cycles for the Newton iteration \eqref{eq:newtonc2pw} where
\[
x =
\frac{1}{2295}\left(\begin{matrix}
-498\\
582
\end{matrix}\right)
,\quad 
y =
\frac{1}{1055}\left(\begin{matrix}
498\\
18
\end{matrix}\right),
\]
and 
\[
z =
\frac{1}{245}\left(\begin{matrix}
102\\
-18
\end{matrix}\right), \quad
w =
-\frac{1}{1405}\left(\begin{matrix}
102\\
582
\end{matrix}\right).
\]
Namely, the cycle equations are satisfied: $(P(x)+T)y=b$, $(P(y)+T)x=b$ and $(P(z)+T)w=b$, $(P(w)+T)z=b$.
Thus, the method has two cycles of order two and since the points are in different quadrants we have that the equation has no solution (see the discussion before Theorem \ref{no-cycles-with-two-elements}).

\end{example}

To define the Jacobi-Newton iteration for solving problem \eqref{eq:pwls}, let us recall the decomposition of the matrix $T$ as a sum of $L+D+U$ where $D$ is the diagonal of $T$ and $L$ and $U$ are the strictly lower and upper parts of $T$, respectively. The iteration is defined as follows: Given $x^0\in\R^n$, 
\begin{equation}\label{jaco_newton}
x^{k+1} := -(P(x^{k})+D)^{-1}(L+U) x^{k} + (P(x^{k})+D)^{-1}b, \quad k\in\NN,
\end{equation}
or equivalently, $x^{k+1}$ is the solution of the diagonal system
\begin{equation}\label{jac_newton2}
(P(x^{k})+D)x^{k+1} = -(L+U) x^{k}+b.
\end{equation}

First, we prove that if the sequence generated by the Jacobi-Newton iteration \eqref{jaco_newton} converges to $\bar x\in\R^n$, then $\bar x$ is a solution of problem \eqref{eq:pwls}.

\begin{proposition}[Limit points of the Jacobi-Newton iteration are solutions]\label{conv-jacobi}
If the sequence $(x^{k})_{k\in\NN}$ generated by \eqref{jaco_newton} converges to $\bar x$, then $\bar x$ solves problem \eqref{eq:pwls}.
\end{proposition}
\begin{proof}
First, we know that $F(x) = x^+ + Tx -b$ is continuous since $x^+$ is a continuous piecewise linear function and $Tx-b$ is linear and therefore continuous. Rewriting $F(x^{k+1})$ and using \eqref{jac_newton2}, we have
\begin{align*}
F(x^{k+1}) &= (P(x^{k+1})+T)x^{k+1}-b,\\
&= (P(x^{k+1})+T)x^{k+1}-(P(x^{k})+D)x^{k+1}-(L+U)x^{k},\\
&= (P(x^{k+1})+D+L+U)x^{k+1}-(P(x^{k})+D)x^{k+1}-(L+U)x^{k},\\
&= (P(x^{k+1})-P(x^{k}))x^{k+1}+(L+U)(x^{k+1}-x^{k}),
\end{align*} using \eqref{jac_newton2} in the last equality. Now, the above equality implies that
$$\|F(x^{k+1})\| \leq \|(P(x^{k+1})-P(x^{k}))x^{k+1}\|+\|L+U\|\|x^{k+1}-x^{k}\|.$$
Hence,
\begin{align*}
\lim_{k \to \infty} \|F(x^{k+1})\| &\leq \lim_{k \to \infty} \|(P(x^{k+1})-P(x^{k}))x^{k+1}\| + \lim_{k \to \infty} \|L+U\|\|x^{k+1}-x^{k}\|\\
&= \lim_{k \to \infty} \|(P(x^{k+1})-P(x^{k}))x^{k+1}\|.
\end{align*}

If $\bar x$ does not have any component equal to $0$, then the result is trivial since the sign becomes constant for a certain index of the sequence. On the other hand, if $\bar x_i = 0$ for some $i$, we have by the continuity of $F$ that
$$ |F_i(\bar x)| = \lim_{k \to \infty} |F_i(x^{k+1})| \leq \lim_{k \to \infty} |(\sgn((x^{k+1}_i)^+)-\sgn((x^{k}_i)^+))x^{k+1}_i| = 0, $$ where the last equality follows from the fact that the $i$-th component of $x^{k+1}$ goes to zero and $\sgn((x^{k+1}_i)^+)-\sgn((x^{k}_i)^+) \in \{-1,0,1\}$. Thus, all component of $F(\bar x)$ are zeroes, which implies the result. 
\end{proof}

Next, we show that $T$ being strongly diagonal dominant is a sufficient condition for the convergence of iteration \eqref{jaco_newton}.

\begin{theorem}[Global convergence of the Jacobi-Newton method]
Let $b \in \R^n$ and $T\in \R^{n\times n}$. If $T$ is strongly diagonal dominant, then the Jacobi-Newton method \eqref{jaco_newton} globally converges to the unique solution of problem \eqref{eq:pwls}.
\end{theorem}

\begin{proof} Let $\bar x$ be the unique solution of problem \eqref{eq:pwls}. Then, $(P(\bar x)+D)\bar x = -(L+U) \bar x+b$, which combined with \eqref{jac_newton2} implies that 
\[(P(x^{k})+D)(x^{k+1}-\bar x) = -(L+U)(x^k-\bar x)+ P(\bar x)\bar x - P(x^{k})\bar x.\]
Hence, 
\begin{align*}
\|x^{k+1}-\bar x\|_{\infty} &= \|-(P(x^{k})+D)^{-1}[(L+U)(x^{k}-\bar x)+(P(\bar x)- P(x^{k}))\bar x]\|_{\infty}\\
&= \max_{i=1,\dots,n}\left|\frac{1}{\sgn((x^{k}_i)^+)+t_{ii}}\left( -\sum_{\substack{j=1 \\ j \neq i}}^n t_{ij}(x^{k}_j-\bar x_j)+\sgn(\bar x_i^+)\bar x_i - \sgn((x^{k}_i)^+)\bar x_i \right)\right|\\
&\leq \max_{i=1,\dots,n}\frac{1}{|\sgn((x^{k}_i)^+)+t_{ii}|}\left( \sum_{\substack{j=1 \\ j \neq i}}^n |t_{ij}||x^{k}_j-\bar x_j|+|\sgn(\bar x_i^+)\bar x_i - \sgn((x^{k}_i)^+)\bar x_i| \right).
\end{align*}
Note that 
\[|\sgn(\bar x_i^+)\bar x_i - \sgn((x^{k}_i)^+)\bar x_i|=\left\{\begin{matrix}0,& \sgn(\bar x_i^+)=\sgn((x^{k}_i)^+),\\|\bar x_i|,&  \sgn(\bar x_i^+)\neq \sgn((x^{k}_i)^+).\end{matrix}\right.\] 
Observe further that $|\sgn(\bar x_i^+)\bar x_i - \sgn((x^{k}_i)^+)\bar x_i|=|\bar x_i|\le |x^{k}_i-\bar x_i|$ when $\sgn(\bar x_i^+)\neq \sgn((x^{k}_i)^+)$ and $|\sgn(\bar x_i^+)\bar x_i - \sgn((x^{k}_i)^+)\bar x_i|=0\le |x^{k}_i-\bar x_i|$ if $\sgn(\bar x_i^+)= \sgn((x^{k}_i)^+)$. Therefore, $|\sgn(\bar x_i^+)\bar x_i - \sgn((x^{k}_i)^+)\bar x_i|\le |x^{k}_i-\bar x_i|$. 

Then,
\begin{align*}
\|x^{k+1}-\bar x\|_{\infty} &\leq \max_{i=1,\dots,n}\frac{1}{|\sgn((x^{k}_i)^+)+t_{ii}|}\left( \sum_{\substack{j=1 \\ j \neq i}}^n |t_{ij}||x^{k}_j-\bar x_j|+|x^{k}_i-\bar x_i| \right)\\
&\leq \left(\max_{i=1,\dots,n}\frac{1}{|\sgn((x^{k}_i)^+)+t_{ii}|}\left( \sum_{\substack{j=1 \\ j \neq i}}^n |t_{ij}|+1 \right)\right)\|x^{k}-\bar x\|_{\infty}\\
&\leq \left(\max_{i=1,\dots,n}\frac{1}{|t_{ii}|}\left( \sum_{\substack{j=1 \\ j \neq i}}^n |t_{ij}|+1 \right)\right)\|x^{k}-\bar x\|_{\infty}.
\end{align*}
The result now follows from the fact that $T$ is strongly diagonal dominant. 
\end{proof}

Now we define the Gauss-Seidel-Newton method as follows: Given $x^0\in\R^n$,
\begin{equation}\label{gs_newton}
x^{k+1} := -(P(x^{k})+D+L)^{-1}U x^{k} + (P(x^{k})+D+L)^{-1}b, \quad k\in\NN,\end{equation}
or equivalently, $x^{k+1}$ is the solution of the triangular linear system
\begin{equation}\label{gs_newton2}
(P(x^{k})+D+L)x^{k+1} = -U x^{k}+b.
\end{equation}

%

\begin{proposition}[Limit points of the Gauss-Seidel-Newton method are solutions]
If the sequence $(x^{k})_{k\in\NN}$ generated by \eqref{gs_newton} converges to $\bar x$, then $\bar x$ solves problem \eqref{eq:pwls}.
\end{proposition}
\begin{proof}
Let $F(x) = x^+ + Tx -b$ and let us rewrite $F(x^{k+1})$ and use \eqref{gs_newton2} to obtain
\begin{align*}
F(x^{k+1}) &= (P(x^{k+1})+T)x^{k+1}-b,\\
&= (P(x^{k+1})+T)x^{k+1}-(P(x^{k})+D+L)x^{k+1}-U x^{k},\\
&= (P(x^{k+1})+D+L+U)x^{k+1}-(P(x^{k})+D+L)x^{k+1}-U x^{k},\\
&= (P(x^{k+1})-P(x^{k}))x^{k+1}+U(x^{k+1}-x^{k}).
\end{align*} The result now follows analogously to Proposition \ref{conv-jacobi}.
\end{proof}

Now, let us show that the strong Sassenfeld's condition gives, besides the existence and uniqueness of a solution, a sufficient condition for global convergence of the Gauss-Seidel-Newton method to the solution.

\begin{theorem}[Global convergence under strong Sassenfeld's condition]\label{glob_conv_GSN}
Let $b \in \R^n$ and $T\in\R^{n\times n}$. If $T$ satisfies the strong Sassenfeld's condition, then the Gauss-Seidel-Newton method \eqref{gs_newton} globally converges to the unique solution of problem \eqref{eq:pwls}.
\end{theorem}
\begin{proof}
Let $\bar x$ be the unique solution of equation \eqref{eq:pwls}. From \eqref{gs_newton2}, we have 
\begin{equation*}
(P(x^{k})+D-L)(x^{k+1}-\bar x) = U(x^{k}-\bar x)+(P(\bar x)-P(x^{k}))\bar x,
\end{equation*} which implies
\begin{equation*}
x^{k+1}_1-\bar x_1 = \frac{1}{\sgn((x^{k}_1)^+)+t_{11}}\left( -\sum_{j=2}^n t_{1j}(x^{k}_j-\bar x_j) + (\sgn(\bar x_1^+)-\sgn((x^{k}_1)^+))\bar x_1 \right),
\end{equation*} then
\begin{align*}
|x^{k+1}_1-\bar x_1| &\leq \frac{1}{|\sgn((x^{k}_1)^+)+t_{11}|}\left( \sum_{j=2}^n |t_{1j}||x^{k}_j-\bar x_j| + |\sgn(\bar x_1^+)-\sgn((x^{k}_1)^+)||\bar x_1| \right)\\
&\leq \frac{1}{|\sgn((x^{k}_1)^+)+t_{11}|}\left( \sum_{j=2}^n |t_{1j}|\|x^{k}-\bar x\|_{\infty} + |x^{k}_1-\bar x_1| \right)\\
&\leq \frac{1}{|\sgn((x^{k}_1)^+)+t_{11}|}\left( 1 + \sum_{j=2}^n |t_{1j}| \right)\|x^{k}-\bar x\|_{\infty}\\
&\leq \beta_1\|x^{k}-\bar x\|_{\infty}.
\end{align*}
By induction let us assume that $|x^{k+1}_i-\bar x_i| \leq \beta_i\|x^{k}-\bar x\|_{\infty}$ for $i=1,\dots,p-1,$ and we will show that $|x^{k+1}_p-\bar x_p| \leq \beta_p\|x^{k}-\bar x\|_{\infty}$.
Now, similarly to the proof of Theorem \ref{exis_uniq_sassenf}, we have
\begin{align*}
x^{k+1}_p-\bar x_p &= \frac{1}{\sgn((x^{k}_p)^+)+t_{pp}}\Biggl( -\sum_{j=1}^{p-1} t_{pj}(x^{k+1}_j-\bar x_j) - \sum_{j=p+1}^n t_{pj}(x^{k}_j-\bar x_j) \\
&\quad + (\sgn(\bar x_p^+)-\sgn((x^{k}_p)^+))\bar x_p \Biggr).
\end{align*}
Taking the absolute value in both sides, we get
\begin{align*}
|x^{k+1}_p-\bar x_p| &\leq \frac{1}{|\sgn((x^{k}_p)^+)+t_{pp}|}\Biggl( \sum_{j=1}^{p-1} |t_{pj}||x^{k+1}_j-\bar x_j| + \sum_{j=p+1}^n |t_{pj}||x^{k}_j-\bar x_j| \\
&\quad + |\sgn(\bar x_p^+)-\sgn((x^{k}_p)^+)\bar x_p| \Biggr),\\ &\leq \frac{1}{|\sgn((x^{k}_p)^+)+t_{pp}|}\left( \sum_{j=1}^{p-1} |t_{pj}||x^{k+1}_j-\bar x_j| + \sum_{j=p+1}^n |t_{pj}||x^{k}_j-\bar x_j| +|x^{k}_p-\bar x_p| \right),\\
&\leq \frac{1}{|\sgn((x^{k}_p)^+)+t_{pp}|}\left(\sum_{j=1}^{p-1}|t_{pj}|\beta_j+\sum_{j=p+1}^n|t_{pj}|+1 \right)\|x^{k}-\bar x\|_{\infty}\\
&\le \beta_p\|x^{k}-\bar x\|_{\infty}.
\end{align*}
Moreover,
\begin{align*}
\| x^{k+1}-\bar x \|_{\infty} &= \max_{i=1,\dots,n.}|x^{k+1}_i-\bar x_i|\\
&\leq \max_{i=1,\dots,n.}\beta_i\|x^{k}-\bar x\|_{\infty}\\
&= \beta\|x^{k}-\bar x\|_{\infty},
\end{align*} and the result follows from the fact that $\beta<1$.
\end{proof}

\section{Computational results} 
\label{sec:computationalresults}

To analyze the three methods and see their differences in practice, we run several examples of problem \eqref{eq:pwls} applying the Jacobi-Newton, Gauss-Seidel-Newton, and the semi-smooth Newton methods. All codes were implemented in Matlab $9.5.0.944444$ (R$2018$b). 

\subsection{Numerical tests for randomly generated data} 
\label{sec:computationalresults}

In this subsection, we work on two groups of problems for the numerical tests. In the first one, we used dense matrices with different dimensions $n$ equal to $1000$, $5000$, and $10000$. In the second group, we considered a matrix with sparse structure and dimensions $1000$, $5000$, and $10000$. The experiments were run on a $2.3$ GHz Intel(R) i5, $16$Gb of RAM, and Windows $10$ operating system. Next, we describe some details about the implementation. 

\medskip

\noindent {\bf (1)} {\it Stopping criterion}: We fix the tolerance for the norm of $F(x^k)$ in \eqref{eq:fucpw} as $10^{-5}$. This means that when the $2$-norm of $(x^{k})^++Tx^{k}-b$ is less than or equal to $10^{-5}$, the execution of the algorithm is stopped, and $x^{k}$ is returned as the solution. The maximum number of iterations was fixed at $1000$, but no problem in our test reached it.

 \medskip

\noindent {\bf (2)} {\it Generating random problems}: To construct the matrices for the first group of experiments, we used the Matlab routine {\it rand} to generate a random dense matrix with a predefined dimension with elements between $-1$ and $1$. To achieve convergence of the different methods, we modified the matrices in order to ensure the validity of the strong diagonal dominance condition in Definition \ref{strong_diag_cond}. To do so, we replaced the diagonal entry with $1.001$ plus the sum of the off-diagonal elements in absolute value; this ensures the convergence of both Jacobi-Newton and Gauss-Seidel-Newton methods. Finally, for the second group of problems, we evoked {\it sprand}, a sparse random matrix generator of Matlab, to generate random sparse matrices with entries between $-1$ and $1$ and density $0.3\%$. The matrices were also similarly modified in order to ensure strong diagonal dominance.

\medskip

\noindent {\bf (3)} {\it Solving linear equations}:  Each iteration of the semi-smooth Newton method requires solving the whole linear system \eqref{eq:newtonc2pw}. On the other hand, the Jacobi-Newton and Gauss-Seidel-Newton methods need to find the solution of a diagonal \eqref{jac_newton2} and upper triangular \eqref{gs_newton2} linear systems, respectively. We used the \emph{backslash} command of Matlab for the dense case since this command uses the diagonal and triangular structures of the matrices generated by both methods \eqref{jac_newton2} and \eqref{gs_newton2}. When solving Newton's linear system for the sparse case, we first use the \emph{symamd} command of Matlab, which makes permutations of rows and columns using the minimum degree algorithm (see \cite{davis2006direct, duff2017direct}) to ensure that the LU-decomposition generates the lowest number of non-zero entries possible.  

\medskip


In order to show the results, we use performance profiles (see \cite{dolan2002benchmarking}). This technique measures and compares the robustness and efficiency of several methods applied to a set of problems using a common indicator, in our case, the CPU time. 
Let us start by presenting the set of tests with dense matrices.

\subsubsection{Piecewise linear equation with dense matrices} \label{densePWL}

We took dense matrices with different dimensions in the first group of problems. We show in Figure \ref{fig:densePerfProf} the results of the performance profiles in our experiments.

\begin{figure}[ht]
\centering

\begin{subfigure}[b]{0.325\textwidth}
\centering
\includegraphics[scale=0.26]{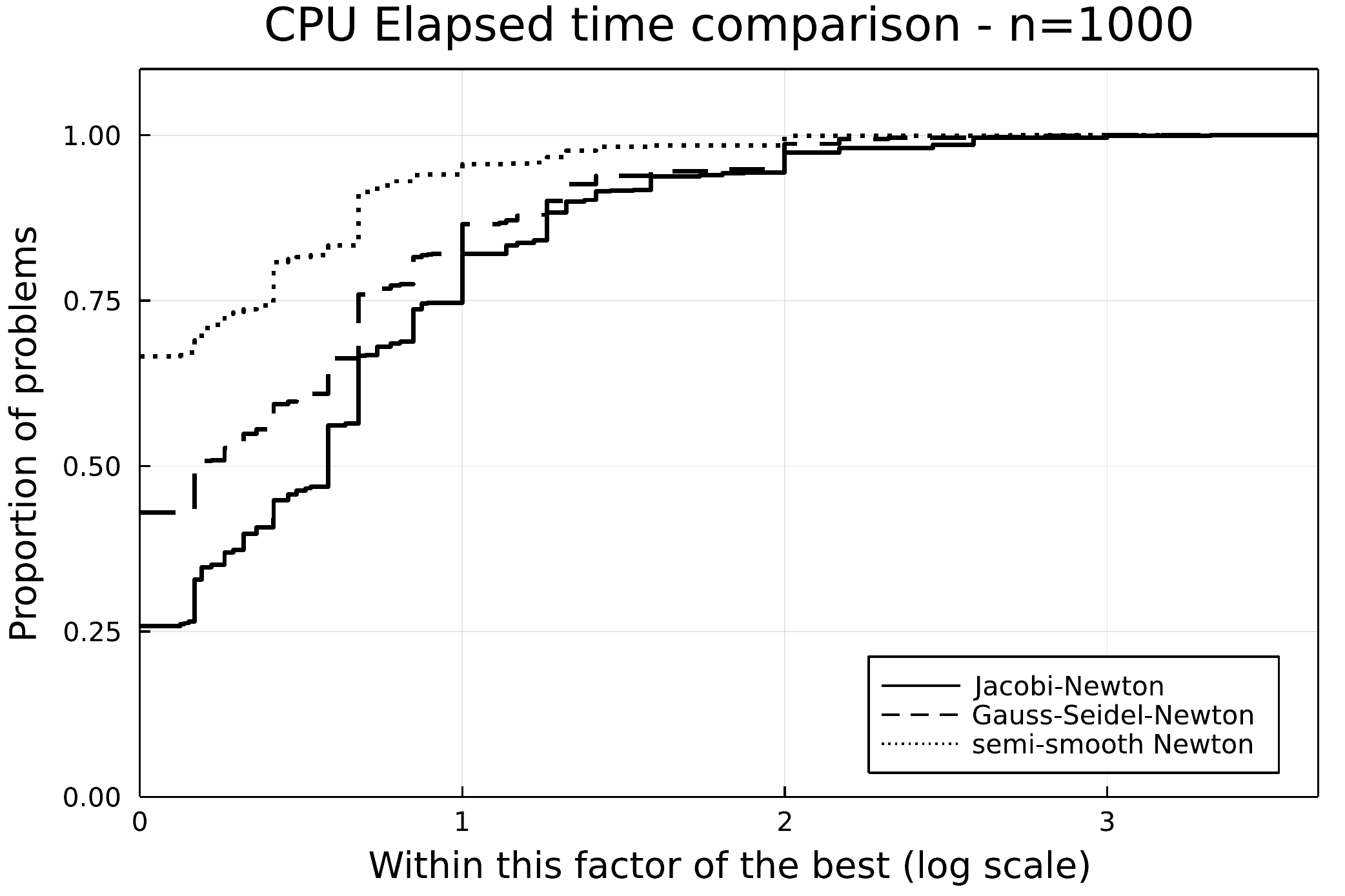}
\caption{Dense matrices $n=1000$}
\label{fig:1000dense}
\end{subfigure}
\hfill
\begin{subfigure}[b]{0.325\textwidth}
\centering
\includegraphics[scale=0.26]{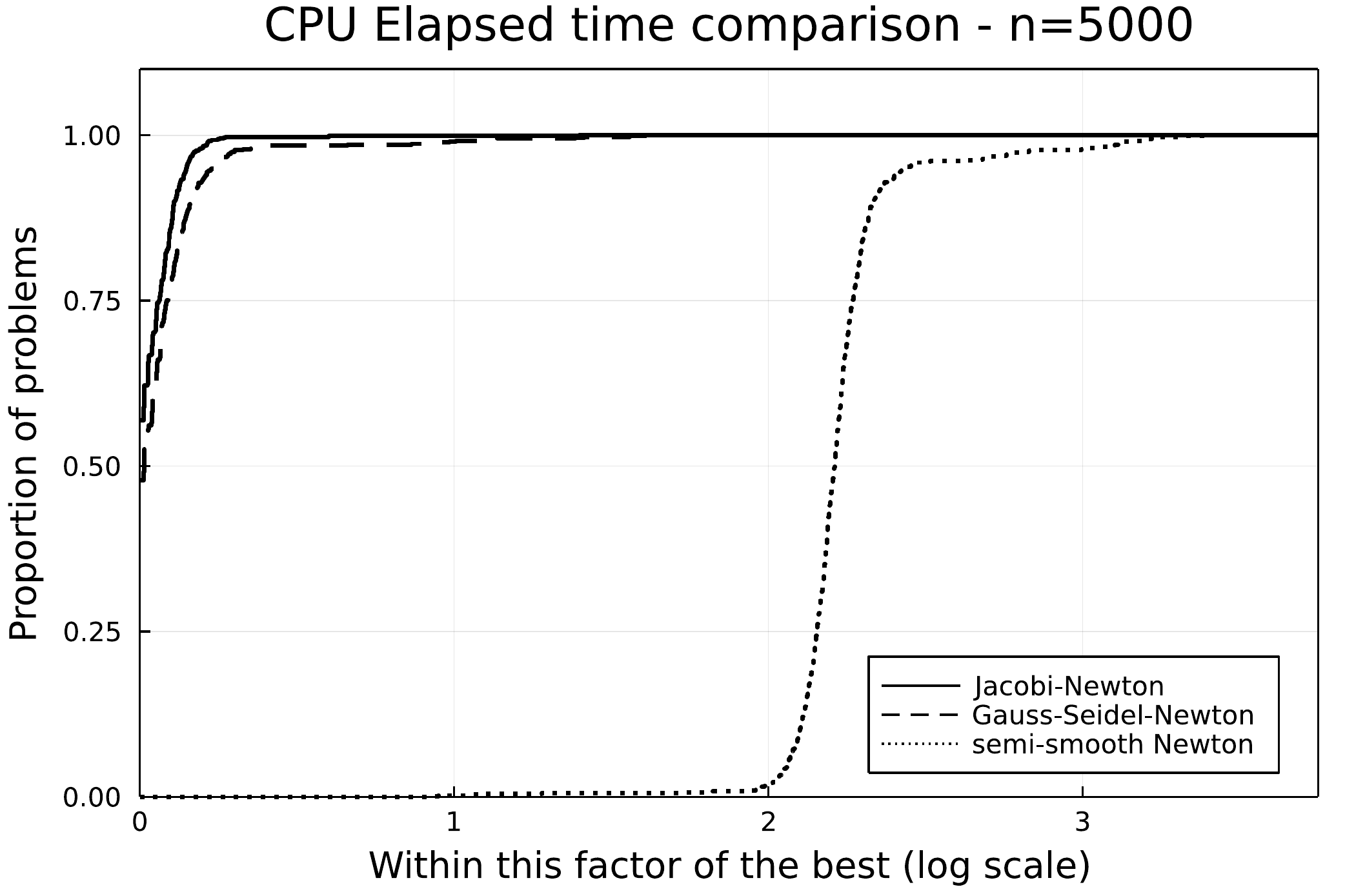}
\caption{Dense matrices $n=5000$}
\label{fig:100dense}
\end{subfigure}
\hfill
\begin{subfigure}[b]{0.325\textwidth}
\centering
\includegraphics[scale=0.26]{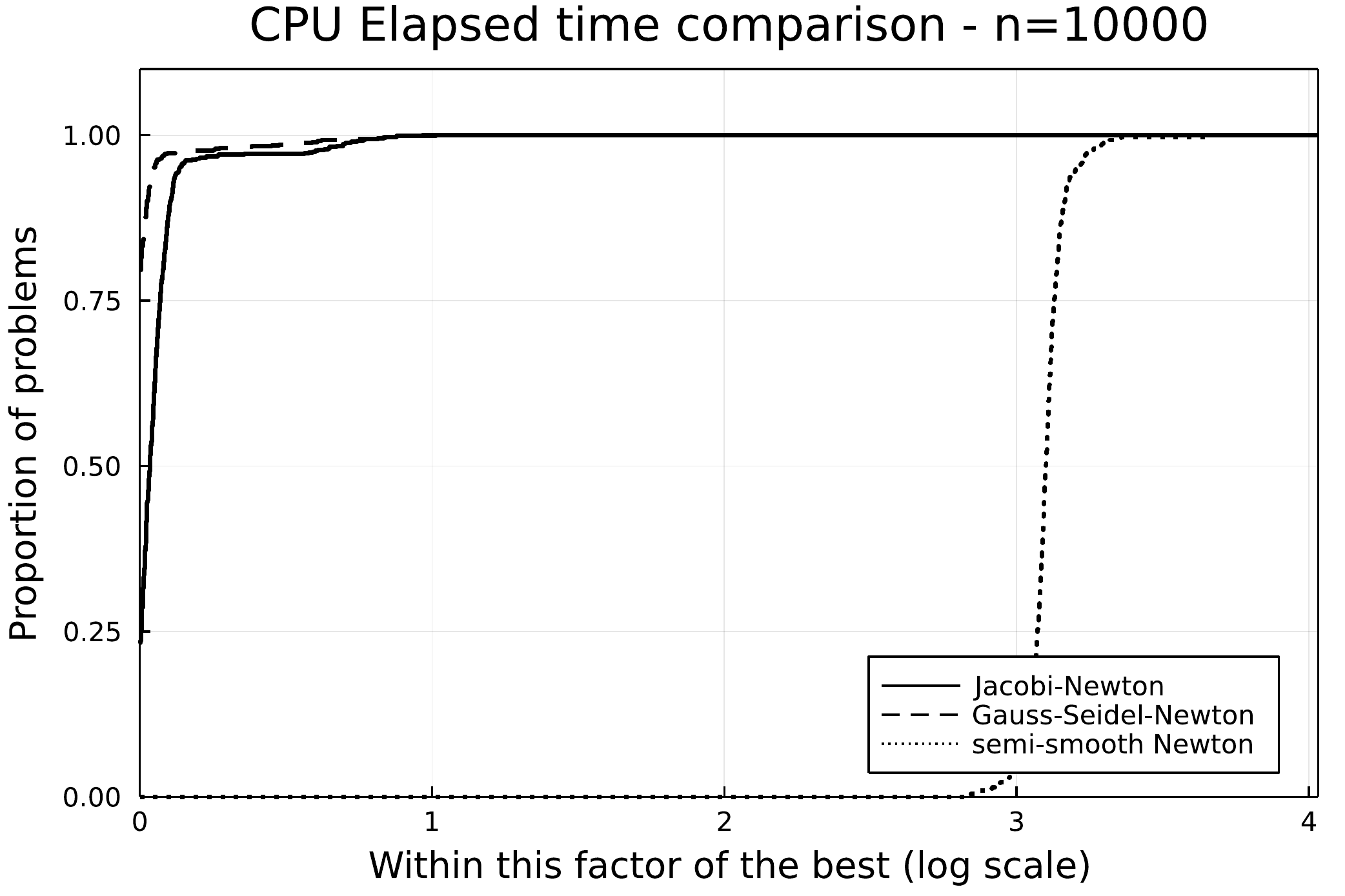}
\caption{Dense matrices $n=10000$}
\label{fig:10000dense}
\end{subfigure}
\caption{Performance profiles for dense matrices in $\log_2$ scale.}
\label{fig:densePerfProf}
\end{figure}

In Figure \ref{fig:densePerfProf}, we compare the robustness and the efficiency of the three methods applied on a set of problems with dimensions $1000$ (low), $5000$ (mid), and $10000$ (high), respectively. The number of problems was fixed at $850$ for each dimension. We first see that for the three sets, every problem was solved by the three methods. Newton's method was the most efficient in the low dimensional test, being the fastest method for circa $70\%$ of the problems. When the dimension increases, Jacobi-Newton and Gauss-Seidel-Newton are much faster. This difference is accentuated in the highest dimension we tested, where also the Gauss-Seidel variant is now slightly better than Jacobi. In particular, Newton's method took at least four times the time the other methods took for almost all mid-dimensional problems. At the same time, it was at least eight times slower for high-dimensional problems. Thus, in our tests, the simplicity of the linear system solved by Gauss-Seidel-Newton and Jacobi-Newton methods (triangular and diagonal, respectively) pays off in comparison with the Newton iteration for mid and high dimensions.

\subsubsection{Piecewise linear equation with sparse matrices} \label{sparsePWL}

Our Jacobi and Gauss-Seidel variants of the Newton iterate were devised mainly with large and sparse problems in mind. In Figure \ref{fig:sparsePerfProf}, we can see the results of our numerical experiments for sparse matrices. Here we also generated $850$ problems for each of the dimensions: $1000$, $5000$, and $10000$.

\begin{figure}[ht]
\centering

\begin{subfigure}[b]{0.325\textwidth}
\centering
\includegraphics[scale=0.26]{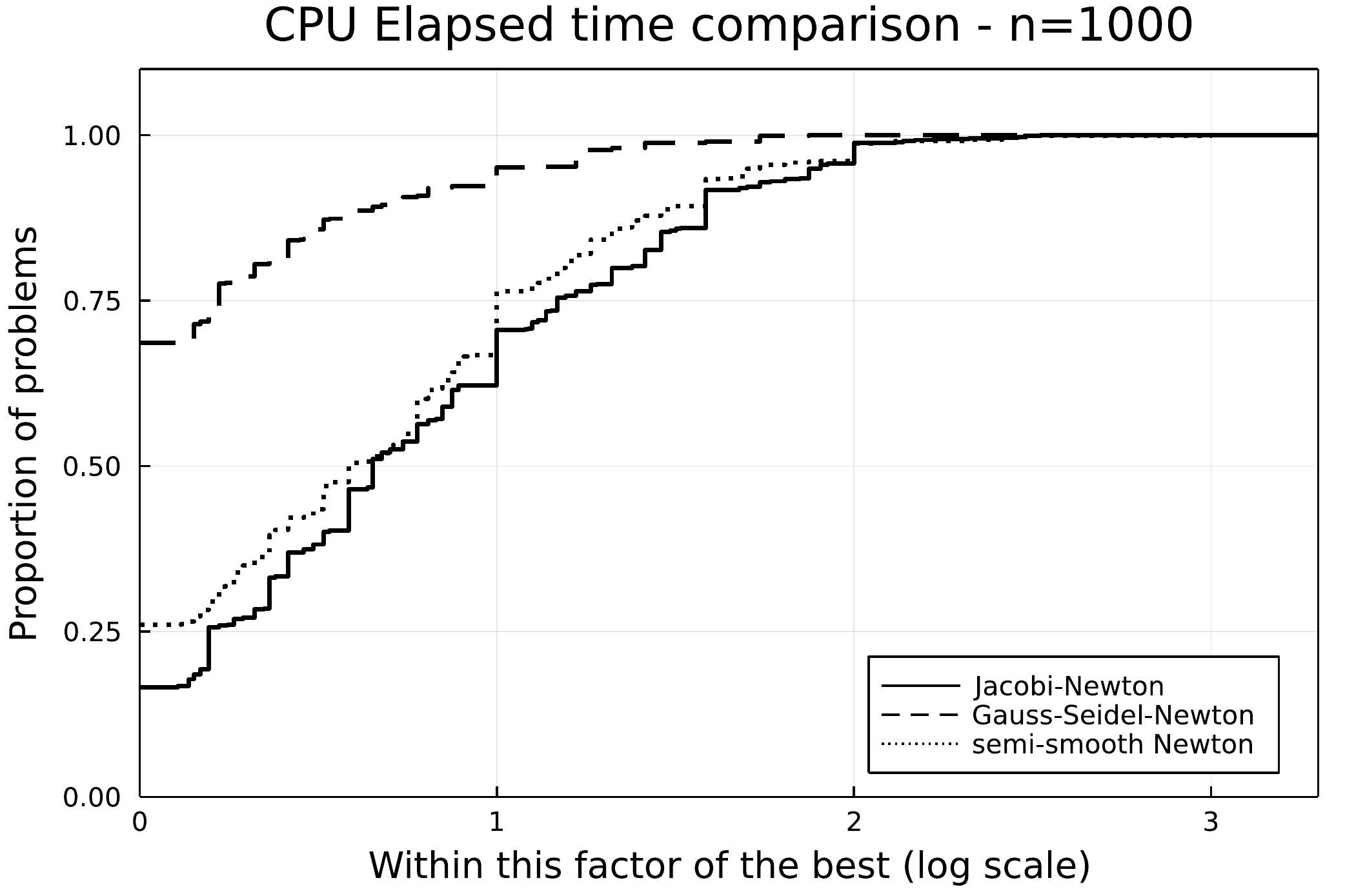}
\caption{Sparse matrices $n=1000$}
\label{fig:1000sparse}
\end{subfigure}
\hfill
\begin{subfigure}[b]{0.325\textwidth}
\centering
\includegraphics[scale=0.26]{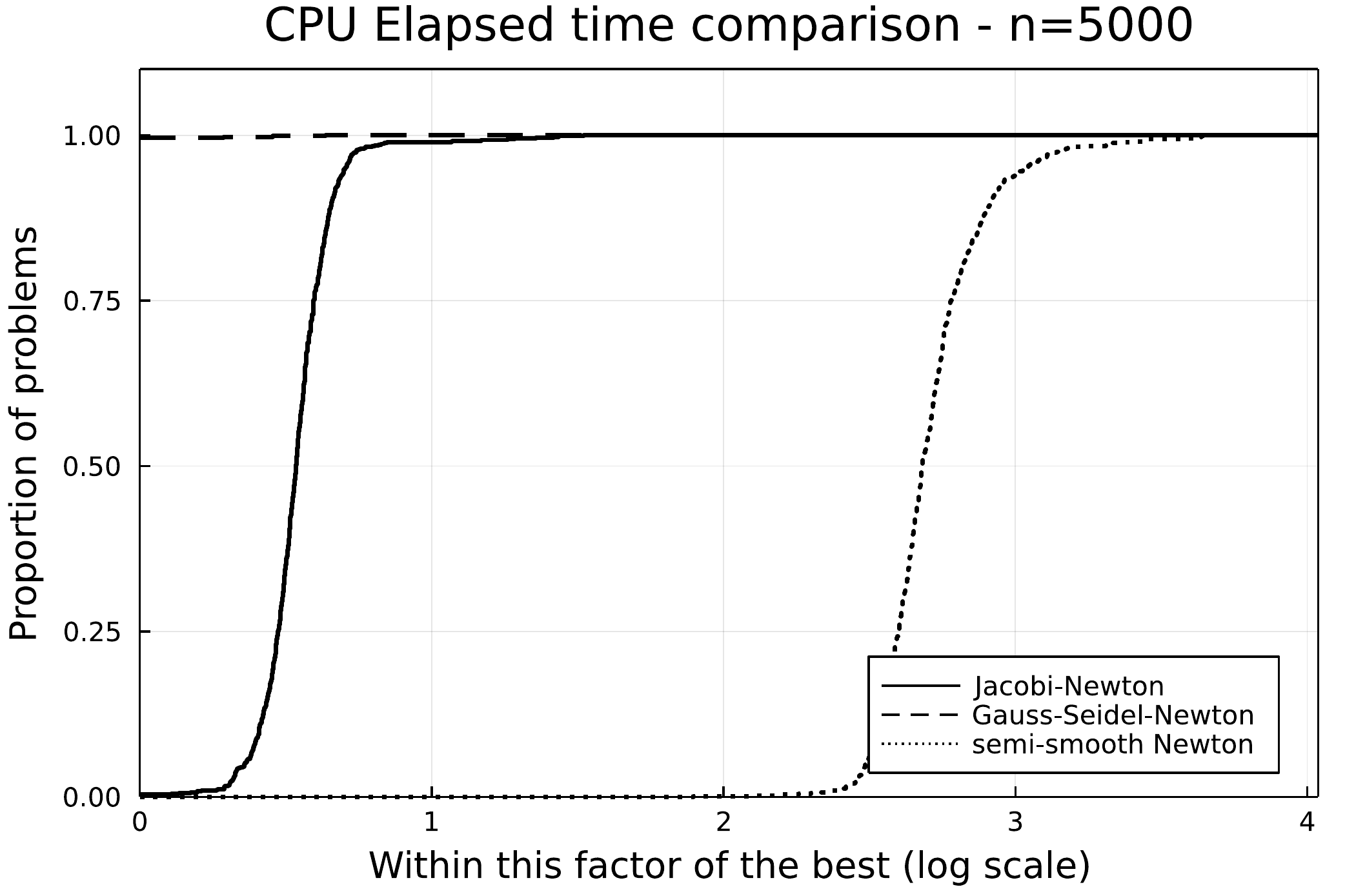}
\caption{Sparse matrices $n=5000$}
\label{fig:5000sparse}
\end{subfigure}
\hfill
\begin{subfigure}[b]{0.325\textwidth}
\centering
\includegraphics[scale=0.26]{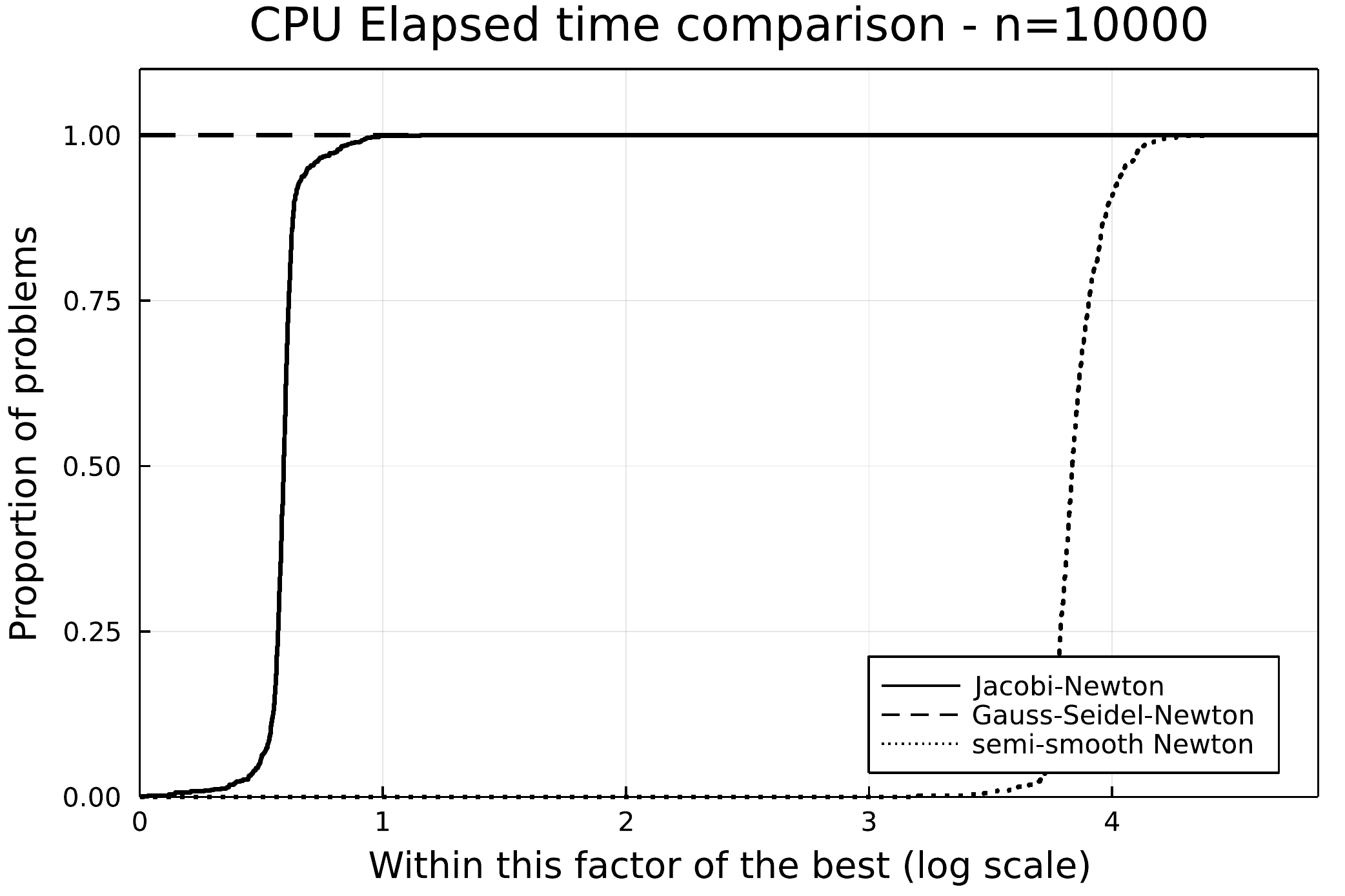}
\caption{Sparse matrices $n=10000$}
\label{fig:10000sparse}
\end{subfigure}
\caption{Performance profiles for sparse matrices in $\log_2$ scale.}
\label{fig:sparsePerfProf}
\end{figure}

In these tests, again all problems were solved by all methods, but here, the superiority of the Gauss-Seidel-Newton iterate is already apparent in the low dimensional test, while Jacobi-Newton and Newton behave similarly. The superiority of Gauss-Seidel-Newton is more evident once the dimension increases, being the fastest method for almost all mid and high-dimensional problems. At the same time, Newton becomes considerably slower than both methods. This behavior was already expected, and they attest that our Gauss-Seidel variant of Newton's method should be the method of choice for large and sparse problems.

\subsection{Application on a discretization of the Boussinesq PDE}
\label{sec:boussinestPDErealdata}

In order to test the proposed methods in solving a real model, we used them to solve an equation studied in \cite{Brugnano2008}. The authors in \cite{Brugnano2008} solve a piecewise linear equation in the form of problem \eqref{eq:pwls} resulting from the discretization of a PDE, the Boussinesq equation, using the semi-smooth Newton method (see \cite{bear1988modeling}), which models a two-dimensional flow of liquid water in a homogeneous phreatic aquifer during seven days. The Boussinesq equation models the water level in time, and a discretization of it results in an equation that can be solved by the semi-smooth Newton, the Jacobi-Newton, and the Gauss-Seidel-Newton methods. After the discretization using a square mesh of size $2N+1$ and the index $l$ representing the respective day, the piecewise linear equation resulting has the following form:
\begin{align}\nonumber
(h_{ij}+\eta_{ij}^{l+1})^+ -& \left(\frac{\kappa \Delta t}{\epsilon}\frac{H_{i,j-\frac{1}{2}}^l}{\Delta y^2}\right)\eta_{ij-1}^{l+1} - \left(\frac{\kappa \Delta t}{\epsilon}\frac{H_{i-\frac{1}{2},j}^l}{\Delta x^2}\right)\eta_{i-1j}^{l+1} \\ \nonumber
+& \frac{\kappa \Delta t}{\epsilon}\left( \frac{H_{i+\frac{1}{2},j}^l+H_{i-\frac{1}{2},j}^l}{\Delta x^2} + \frac{H_{i,j+\frac{1}{2}}^l+H_{i,j-\frac{1}{2}}^l}{\Delta y^2} \right) \eta_{ij}^{l+1} - \left(\frac{\kappa \Delta t}{\epsilon}\frac{H_{i+\frac{1}{2},j}^l}{\Delta x^2}\right)\eta_{i+1j}^{l+1}\\
-& \left(\frac{\kappa \Delta t}{\epsilon}\frac{H_{i,j+\frac{1}{2}}^l}{\Delta y^2}\right)\eta_{ij+1}^{l+1} = H_{ij}^l + \frac{\Delta t}{\epsilon}\phi_{ij}^l. \label{discretization}
\end{align}
The parameters $\epsilon:=0.4$ and $\kappa:=1$ are the porosity and the hydraulic conductivity of the aquifer, respectively, while $\Delta x:=\frac{L}{N}, \Delta y:=\frac{L}{N}, \Delta t:=86400$s ($1$ day) are the stepsizes in the $x$-axis, $y$-axis, and the time step, respectively, where the aquifer is assumed to be a paraboloid of revolution with maximum radius $L:=1000$ meters and depth of $10$ meters. Water sinks from a pointwise source located at the bottom of the aquifer, which corresponds to the parameter $\phi_{00}^l := -\frac{q}{\Delta x \Delta y}$, for every $l$, where $q:=10$ m$^3$/s is the volume of water sinking per second, with $\phi_{ij}^l:=0$ for the remaining indexes $i$ and $j$. The variables of the system \eqref{discretization} are $\eta_{ij}^{l+1}$, $l=0,\dots,6$ and $i,j = -N,\dots, N$, which represents the distance from the reference level $0$ to the surface of water at the corresponding point in space and time, while $h_{ij}$ is the given distance from the bottom of the aquifer to the reference level $0$ at the correspondent point in space. Finally, the parameter $H^l_{ij}:=(h_{ij}+\eta^l_{ij})^+$ is the total distance from the bottom of the aquifer to the surface level, which is obtained from the solution $\eta^l_{ij}$ computed in the previous day, where the level of water at day zero sits at the reference level $0$. The terms $H^l_{i\pm\frac{1}{2},j}$ and $H^l_{i,j\pm\frac{1}{2}}$ are defined as the averages of their nearest grid values. We refer to \cite{Brugnano2008} for a more detailed description of this model. In order to write the system for the day $l+1$ in the format of \eqref{eq:pwls}, we considered the change of variables $x_{ij} = h_{ij}+\eta_{ij}^{l+1}$. 
 The resulting piecewise linear system has a symmetric, positive semidefinite and block-tridiagonal matrix, specifically, the diagonal blocks are tridiagonal matrices, and the subdiagonal blocks are diagonal, so it has a pseudo pentadiagonal structure, meaning that we only need three vectors to save the entire matrix. This structure is exploited in order to compute the iterates of each method. For solving the system for each day $l+1=1,2,\ldots,7$, we used the Gauss-Seidel-Newton method since it showed better results for sparse large-scale matrices. Actually, Gauss-Seidel-Newton was about three times faster than Jacobi-Newton in an initial test with $N=50$. We used a tolerance of $10^{-5}$ for the norm of $F(x^k)$ in \eqref{eq:fucpw} to stop the iteration.

A significant detail in our implementation is the choice of the initial point for the method at each day. Having solved the system for the grid size $N=50$, we used an interpolated version (completing the missing nodes with a mean using the nearest values) of that solution for smaller grid sizes, specifically $N=100$. Then, we applied the same strategy for $N=200$ (using the levels for $N=100$). It is worth noting that this strategy helps the Gauss-Seidel-Newton method to converge faster, and it is different from the one used in \cite{Brugnano2008} where the water level for the previous day is used as the initial point. Similarly, the solution for $N=50$ is obtained from the solution of $N=25$, where the initial point from \cite{Brugnano2008} was used for $N=25$.

In table \ref{tab:BrugnaroExpResults}, we present the results in terms of time (in minutes) for each day, needed to solve each system (time), and the approximated volume of water in the phreatic aquifer at each day computed with the solution found for the equation. This serves as a measure of accuracy of the solution found as this volume can be computed from the model, that is, at day $0$, the volume of water is equal to the volume of the aquifer times the porosity constant $\kappa$, which amounts to $2,000,000\pi$ cubic meters. Thus, considering the constant flow of water $q$, we can predict the total volume of water at day $7$ to be approximately $235,185.3$ cubic meters, which is well approximated by the solution found with $N=200$. Note that our computed volume coincides with the one computed in \cite{Brugnano2008}.

\begin{table}[h!]
\centering
    \begin{tabular}{|c|c|c|c|c|c|c|c|c|c|} 
    \hline
        day$\setminus$N & \multicolumn{2}{|c|}{50} & \multicolumn{2}{|c|}{100} & \multicolumn{2}{|c|}{200} \\ \hline
        -      
        & time & volume 
        & time & volume 
        & time & volume \\ \hline
        1         
        & 9.06 & 5,419,110.3 
        & 103.64 & 5,419,172.7 
        & 1561.03 & 5,419,182.2 \\
        2         
        & 8.27 & 4,555,110.2 
        & 93.80 & 4,555,172.7 
        & 1393.78 & 4,555,182.2 \\
        3         
        & 7.58 & 3,691,110.1 
        & 82.92 & 3,691,172.6 
        & 1273.69 & 3,691,182.1   \\
				4         
        & 7.25 & 2,827,109.9 
        & 72.98 & 2,827,172.6 
        & 1155.79 & 2,827,182.1   \\
				5         
        & 6.24 & 1,963,109.8 
        & 64.82 & 1,963,172.5 
        & 934.44 & 1,963,182.1   \\
				6         
        & 5.20 & 1,099,109.8 
        & 55,04 & 1,099,172.5 
        & 834.09 & 1,099,182.1   \\
				7         
        & 5.71 & 235,109.7 
        & 57.60 & 235,172.4 
        & 806.94 & 235,182.0   \\
    \hline
    \end{tabular}
    \caption{Results for discretizations $N=50, 100, 200$ for the Gauss-Seidel-Newton method. 
    }
		\label{tab:BrugnaroExpResults}
 \end{table}


 \begin{figure}[h!]
\centering
\includegraphics[scale=0.4]{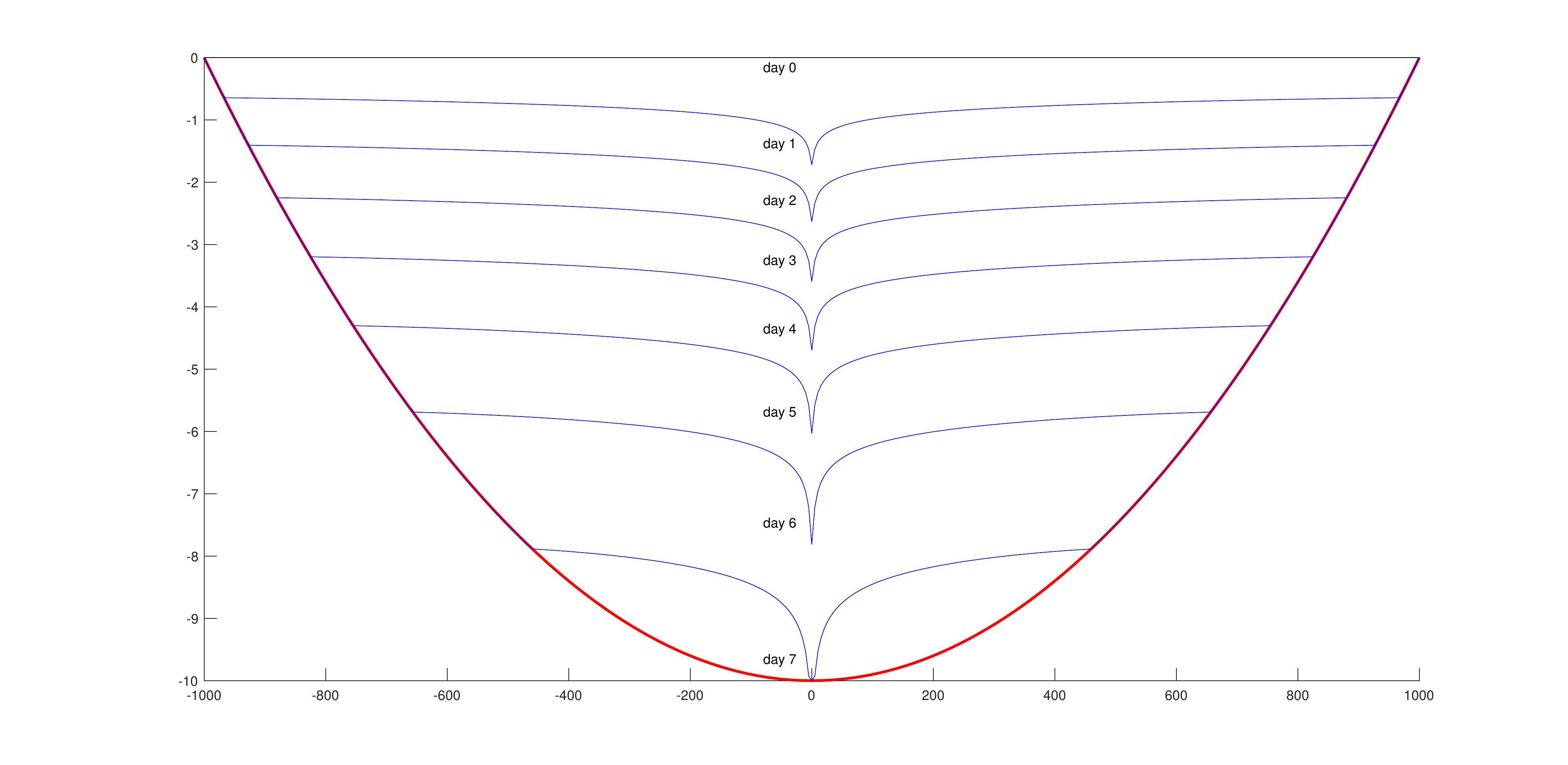}
\caption{Level of water at days $0,1,\dots,7$ leaking from a pointwise sink at the bottom of a paraboloid phreatic aquifer.}
\label{fig:N200levels}
\end{figure}

An important remark is the fact that the matrix defining the problem does not satisfy the sufficient condition for the global convergence of Gauss-Seidel-Newton (Theorem \ref{glob_conv_GSN}), which slows down the methods considerably in comparison with the standard Newton iterate. In particular our Newton implementation, as described previously, was about $20$ times faster on average for this problem, requiring only $3$ or $4$ iterations. Nevertheless, both our proposed methods are still able to converge, showing robustness, which is the point of this numerical experiment. For comparison, at the cost of one Newtonian iteration for $N=50$, we are able to compute approximately $1000$ iterations of the Gauss-Seidel-Newton method or approximately $3000$ iterations of the Jacobi-Newton method. For the Jacobi-Newton iterate one can easily make use of parallel computations to speed up the algorithm. Finally, in Figure \ref{fig:N200levels}, we draw a two-dimensional vertical cut of the approximated water levels found for each day with the Gauss-Seidel-Newton method for discretization $N=200$.



\section{Conclusions} \label{sec:conclusions}

In this paper, we considered iterative schemes for solving the piecewise linear equation $x^++Tx=b$, where $x^+$ denotes projection onto the non-negative orthant. This problem appears in solving absolute value equations and minimizing a quadratic function over the non-negative orthant. A semi-smooth Newton method has been proposed for this problem, where the existence and uniqueness of solutions were studied together with the finite convergence of the method. In \cite{Bello-Cruz:2017}, the authors conjecture that positive definiteness of $T$ would be sufficient for finite convergence of the semi-smooth Newton method. However, we showed that this assumption is enough only to avoid cycles of size two in general. 

To avoid solving a full linear system of equations at each Newtonian iteration, we proposed Newtonian methods inspired by the classical Jacobi and Gauss-Seidel methods for linear equations, where only a diagonal or triangular linear system is solved at each iteration. The existence and uniqueness of solutions are shown together with global convergence of the methods under stronger variants of the well-known sufficient conditions of convergence for linear systems, namely, diagonal dominance (for the Jacobi iterate) and Sassenfeld's criterion (for the Gauss-Seidel iterate). Numerical experiments were conducted on random problems to attest that the methods are comparable with the standard Newtonian approach, being considerably faster for large-scale and sparse problems. In an applied experiment concerning the discretization of the Boussinesq equation, we show that both methods are robust and reliable even when sufficient conditions for convergence are not met.

For future work, we expect to address the possibility of weakening the sufficient conditions we obtained for the global convergence of the Jacobi-Newton and Gauss-Seidel-Newton iterations. For instance, extensive numerical experiments suggest that Gauss-Seidel-Newton converges globally when $T$ is a symmetric and positive definite matrix. Another possibility would be to combine Jacobi and Gauss-Seidel iterates in an SOR-style, which may produce interesting theoretical and numerical results. Additionally, instead of considering projection onto the non-negative orthant, we expect to address the analogous equations obtained by projecting onto the second-order cone or the semidefinite cone. The situation is more challenging as the projection matrices in those cases do not have such a simple diagonal structure.

\bibliographystyle{ieeetr}
\bibliography{Newton-References}

\end{document}